\documentclass[reqno,11pt]{amsart}

\usepackage{etex}
\usepackage{amsmath, amssymb}
\usepackage{array}
\usepackage[frame,cmtip,arrow,matrix,line,graph,curve]{xy}
\usepackage{graphpap, color, paralist, pstricks}
\usepackage[mathscr]{eucal}
\usepackage[pdftex]{graphicx}
\usepackage[pdftex,colorlinks,backref=no,citecolor=blue]{hyperref}
\usepackage{tikz}

\numberwithin{equation}{section}

\newtheorem{thm}{Theorem}[section]
\newtheorem{prop}[thm]{Proposition}
\newtheorem{cor}[thm]{Corollary}
\newtheorem{lem}[thm]{Lemma}

\theoremstyle{plain}

\newtheorem{rem}[thm]{Remark}

\newcommand{\cc}{\mathbb{C}}

\newcommand{\nn}{\mathbb{N}}
\newcommand{\pp}{\mathbb{P}}
\newcommand{\qq}{\mathbb{Q}}
\newcommand{\zz}{\mathbb{Z}}

\newcommand{\rS}{\mathit{S}}
\newcommand{\SL}{\mathrm{SL}}
\newcommand{\End}{\mathrm{End}\,}

\newcommand{\EEnd}{{\mathscr E\!nd}\,}
\newcommand{\Hom}{\mathrm{Hom}}

\newcommand{\coeff}{\mathrm{Coeff}}

\newcommand{\con}{\mathrm{con}}
\newcommand{\dest}{\mathrm{dest}}
\newcommand{\fr}{\mathrm{fr}}
\newcommand{\id}{\mathrm{id}}
\newcommand{\im}{\mathrm{im \,}}

\newcommand{\rk}{\mathrm{rank} \,}
\newcommand{\Pic}{\mathrm{Pic}}

\newcommand{\bt}{\mathbf{t}}

\newcommand{\bM}{\mathbf{M}}
\newcommand{\bP}{\mathbf{P}}
\newcommand{\bR}{\mathbf{R}}
\newcommand{\bU}{\mathbf{U}}
\newcommand{\bFM}{\mathbf{FM}}
\newcommand{\bFH}{\mathbf{FH}}
\newcommand{\bOFH}{\mathbf{OFH}}
\newcommand{\bOFM}{\mathbf{OFM}}

\newcommand{\cH}{\mathcal{H}}
\newcommand{\cJ}{\mathcal{J}}
\newcommand{\cL}{\mathcal{L}}

\newcommand{\cN}{\mathcal{N}}
\newcommand{\cO}{\mathcal{O} }
\newcommand{\cP}{\mathcal{P} }

\newcommand{\fE}{\mathfrak{E} }
\newcommand{\fP}{\mathfrak{P} }
\newcommand{\fQ}{\mathfrak{Q} }
\newcommand{\fN}{\mathfrak{N} }
\newcommand{\fR}{\mathfrak{R} }
\newcommand{\fT}{\mathfrak{T} }
\newcommand{\fU}{\mathfrak{U} }

\newcommand{\tbFM}{\widetilde{\mathbf{FM}}}
\newcommand{\tbFH}{\widetilde{\mathbf{FH}}}

\newcommand{\tp}{\tilde{p}}
\newcommand{\ve}{\varepsilon }
\newcommand{\s}{\sigma }

\newcommand{\dss}{\displaystyle}

\begin{document}

\title[A chain of $\cc^{*}$-flips of moduli of $\cO$-twisted rank 2 constrained framed Hitchin pairs on a curve]{A chain of $\cc^{*}$-flips of the moduli spaces of $\cO$-twisted rank 2 constrained framed Hitchin pairs on a smooth curve}

\author{YongJoo Shin}
\address{Department of Mathematics, Chungnam National University,
Science Building 1, 99 Daehak-ro, Yuseong-gu, Daejeon 34134, Republic of Korea}
\email{haushin@cnu.ac.kr}

\author{Sang-Bum Yoo}
\address{Department of Mathematics Education, Gongju National University of Education, 27 Ungjin-ro, Gongju-si, Chungcheongnam-do, 32553, Republic of Korea}
\email{sbyoo@gjue.ac.kr}


\begin{abstract}	
Let $X$ be a smooth complex projective curve. We prove that there exists a surjective commutative forgetful diagram from the chain of $\cc^{*}$-flips of the moduli spaces of $\cO_{X}$-twisted rank 2 constrained framed Hitchin pairs on $X$ to the chain of $\cc^{*}$-flips of the moduli spaces of rank 2 framed modules on $X$.
\end{abstract}

\maketitle

\section{Introduction}
Birational geometries of the moduli spaces of Hitchin pairs were studied in various ways, using the symplectic cut \cite{Haus01}, the flips of moduli of parabolic Higgs bundles \cite{Tha02} and the flips of moduli of framed Hitchin pairs \cite{Sch00}. Especially, the second and third ways can be recovered by using GIT flips of \cite{Tha96}. We focus on the third.

Let $X$ be a smooth complex projective curve of genus $g$ throughout this paper. Let $L$ and $\cH$ be vector bundles on $X$. Fix a positive $\s\in\qq$ . Then we consider the moduli space of $\s$-semistable (resp. $\s$-stable) framed modules $(E,\psi\colon E\to\cH)$ of rank $r$ and degree $d$ on $X$ denoted by $\bFM^{\s-ss}(r,d,\cH)$ (resp. $\bFM^{\s-s}(r,d,\cH)$), and the moduli space of $\s$-semistable (resp. $\s$-stable) $L$-twisted framed Hitchin pairs $(E,\ve\in\cc,\phi\colon E\to E\otimes L,\psi\colon E\to\cH)$ of rank $r$ and degree $d$ on $X$ denoted by $\bFH^{\s-ss}(r,d,L,\cH)$ (resp. $\bFH^{\s-s}(r,d,L,\cH)$). Further, $\bFH_{\con}^{\s-ss}(r,d,L,\cH)$ (resp. $\bFH_{\con}^{\s-s}(r,d,L,\cH)$) denotes the locus of $\bFH^{\s-ss}(r,d,L,\cH)$ (resp. $\bFH^{\s-s}(r,d,L,\cH)$) parametrizing $(E,\ve,\phi,\psi)$ such that $\phi\in\End((E,\psi))$, where $\End((E,\psi))$ is the space of endomorphisms of $(E,\psi)$ (see Section \ref{s2}). The object parametrized by $\bFH_{\con}^{\s-ss}(r,d,L,\cH)$ (resp. $\bFH_{\con}^{\s-s}(r,d,L,\cH)$) is called a $\s$-semistable $L$-twisted (resp. $\s$-stable $L$-twisted) {\bf constrained} framed Hitchin pair.

A chain of $\cc^{*}$-flips connecting to the moduli space of vector bundles on $X$ was constructed via a variation of $\bFM^{\s-ss}(r,d,\cH)$ in \cite{OST99}. A chain of $\cc^{*}$-flips connecting to the moduli space of $L$-twisted Hitchin pairs on $X$ was constructed via a variation of $\bFH^{\s-ss}(r,d,L,\cH)$ in \cite{Sch00}. Here a $\cc^{*}$-flip is a variation of GIT under a $\cc^{*}$-action in the sense of \cite[Theorem 3.3]{Tha96}. The purpose of this work is to show the following main result.

\begin{thm}[{i.e. Theorem \ref{commutativity of C star filps}}]\label{main thm}
There exists a surjective commutative forgetful diagram from the chain of $\cc^{*}$-flips of $\bFH_{\con}^{\s-ss}(2,d,\cO_{X},\cH)$ to the chain of $\cc^{*}$-flips of $\bFM^{\s-ss}(2,d,\cH)$.
\end{thm}

To prove Theorem \ref{main thm} we need to consider equivalences of stabilities between $\cO_{X}$-twisted rank $2$ constrained (oriented) framed Hitchin pairs and (oriented) framed modules on $X$ as follows. See Section \ref{stabilities} for details.

\begin{thm}[{i.e. Theorem \ref{ss FH equiv ss FM}, and cf. Theorem \ref{ss OFH equiv ss OFM} for an oriented case}]\label{ss O-twisted framed Hitchin pair}
A constrained $\mathcal{O}_X$-twisted framed Hitchin pair is $\s$-semistable if and only if its underlying framed module is $\s$-semistable.
\end{thm}

\begin{thm}[{i.e. Corollary \ref{st FH equiv st FM in rank 2}, and cf. Theorem \ref{st OFH equiv st OFM in rank 2} for an oriented case}]\label{ss O-twisted rank 2 framed Hitchin pair}
A constrained $\mathcal{O}_X$-twisted rank $2$ framed Hitchin pair with a nonzero framing is $\s$-stable if and only if its underlying framed module is $\s$-stable.
\end{thm}

\begin{rem}
Theorems \ref{ss O-twisted framed Hitchin pair} and \ref{ss O-twisted rank 2 framed Hitchin pair} hold for a smooth complex projective variety $Y$ of arbitrary dimension. Indeed, except for the statements of the moduli of oriented framed Hitchin pairs, the sections \ref{s2} and \ref{stabilities} are valid for such $Y$ of any dimension.
\end{rem}
This paper is organized as follows. In Section \ref{s2}, we introduce basics on framed modules, framed Hitchin pairs and their oriented objects. In Section \ref{flips}, we provide the chain of $\cc^{*}$-flips of $\bFM^{\s-ss}(r,d,\cH)$ by \cite{OST99} and that of $\bFH^{\s-ss}(r,d,L,\cH)$ by \cite{Sch00}. In Section \ref{stabilities}, we show that the stability of $\bFH_{\con}^{\s-ss}(2,d,\cO_{X},\cH)$ corresponds to the stability of $\bFM^{\s-ss}(2,d,\cH)$. In Section \ref{cstar flips}, we give a proof of Theorem \ref{main thm} (i.e. Theorem \ref{commutativity of C star filps}). In Section \ref{application}, we discuss the surjective commutative forgetful diagram from the chain of $\cc^{*}$-flips of $\bFH_{\con}^{\s-ss}(2,d,\cO_{X},\cO_{X})$ to the chain of $\cc^{*}$-flips of $\bFM^{\s-ss}(2,d,\cO_{X})$ of Theorem \ref{main thm} more explicitly.

\section{Framed modules, framed Hitchin pairs and their oriented objects}\label{s2}
In this section we give definitions of (oriented) framed modules, (orinted) framed Hitchin pairs and their stabilities. See \cite{HL95b, OST99, Sch00} for more details.
\subsection{Framed modules}\label{framed modules}
Let $\cH$ be a vector bundle on $X$. {\bf A framed module of type $(r,d,\cH)$ on $X$} is a pair $(E,\psi)$ which consists of the following ingredients:
\begin{itemize}
\item a vector bundle $E$ of rank $r$ and degree $d$ on $X$;
\item a homomorphism $\psi\colon E\to\cH$, called the framing of $E$.
\end{itemize}

Fix a positive $\s\in\qq$. For a vector bundle $E$, $P_{E}$ denotes the Hilbert polynomial of $E$. The Hilbert polynomial of a framed module $(E,\psi)$ is defined by $P_{(E,\psi)}:=P_{E}-\delta(\psi)\cdot\s$, where $\delta(\psi)=1$ if $\psi\ne0$ and $\delta(\psi)=0$ otherwise.

If $F$ is a subbundle of $E$ with the quotient $G=E/F$, then a framing $\psi\colon E\to\cH$ induces framings $\psi_{F}=\psi|_{F}$ on $F$ and $\psi_{G}$ on $G$: if $\psi_{F}\ne0$, then $\psi_{G}=0$, and otherwise it is the induced homomorphism of $\psi$ on $G$. With this convention we have
$$P_{(E,\psi)}=P_{(F,\psi_{F})}+P_{(G,\psi_{G})}.$$

A framed module $(E,\psi)$ is {\bf$\s$-(semi)stable}, if for any proper, nonzero subbundle $F$ of $E$,
$$\frac{P_{(F,\psi_{F})}}{\rk F}<(\le)\frac{P_{(E,\psi)}}{\rk E}.$$

Let $\bFM^{\s-ss}(r,d,\cH)$ be the moduli space of $\s$-semistable framed modules of type $(r,d,\cH)$ with nonzero framings on $X$. $\bFM^{\s-ss}(r,d,\cH)$ was constructed in  \cite{HL95b}. $\bFM^{\s-s}(r,d,\cH)$ denotes the stable locus of $\bFM^{\s-ss}(r,d,\cH)$.

\begin{rem}
The construction of $\bFM^{\s-ss}(r,d,\cH)$ in \cite{HL95b} has a generality. The moduli space parametrizes objects with underlying coherent sheaves on a smooth complex projective variety, where $\s$ is a rational polynomial with a positive leading coefficient.
\end{rem}

Let $(E,\psi)$ and $(E',\psi')$ be framed modules of type $(r,d,\cH)$ and $(r',d',\cH')$ respectively. {\bf A homomorphism $h\colon (E,\psi)\to(E',\psi')$ of framed modules} is a homomorphism of the underlying bundles $h\colon E\to E'$ for which there is an element $w\in\cc$ such that $\psi'\circ h=w\psi$. The space $\Hom((E,\psi),(E',\psi'))$ of homomorphisms of framed modules is a linear subspace of $\Hom(E,E')$. The space $\Hom((E,\psi),(E,\psi))$ is denoted by $\End((E,\psi))$, called {\bf the space of endomorphisms of $(E,\psi)$}. If {\bf $h\colon (E,\psi)\to(E',\psi')$ is an isomorphism}, then it is an isomorphism of the underlying bundles $h\colon E\to E'$ and the factor $w$ is taken in $\cc^{*}$. In particular, the isomorphism $h_{0}:=w^{-1}h$ satisfies $\psi'\circ h_{0}=\psi$.

\begin{lem}\label{vanishing hom}
Let $(E,\psi)$ and $(E',\psi')$ be $\s$-semistable framed modules with framings $\psi\colon E\to\cH$ and $\psi'\colon E'\to \cH'$. If $\dss\frac{P_{(E,\psi)}}{\rk E}>\frac{P_{(E',\psi')}}{\rk E'}$, then $\Hom((E,\psi),(E',\psi'))=0$. If $(E,\psi)$ and $(E',\psi')$ are $\s$-stable with $\dss\frac{P_{(E,\psi)}}{\rk E}=\frac{P_{(E',\psi')}}{\rk E'}$, then any nontrivial homomorphism $h\colon (E,\psi)\to(E',\psi')$ is an isomorphism. Moreover in this case $\Hom((E,\psi),(E',\psi'))\cong\cc$.
\end{lem}
\begin{proof}
Let's prove the first statement. Suppose that $h\colon E\to E'$ is nontrivial such that there is an element $w\in\cc$ such that $\psi'\circ h=w\psi$. Let $G:=\im h$. Then we have
$$\frac{P_{(G,\psi'_{G})}}{\rk G}\le\frac{P_{(E',\psi')}}{\rk E'}$$
and
$$\frac{P_{(G,\psi_{G})}}{\rk G}\ge\frac{P_{(E,\psi)}}{\rk E}.$$
Note that $\psi'(h(e))=w\psi(e)=w\psi_{G}(h(e))$ for any $e\in E$. If $w=0$, then $\psi'_{G}=0$ and so $P_{(G,\psi_{G})}\le P_{G}=P_{(G,\psi'_{G})}$. If $w\ne0$, then $\psi'_{G}=0$ implies $\psi_{G}=0$, and so $P_{(G,\psi_{G})}\le P_{(G,\psi'_{G})}$. Hence we have $$\dss\frac{P_{(E,\psi)}}{\rk E}\le\frac{P_{(E',\psi')}}{\rk E'},$$ which is a contradiction to the assumption.

Next we prove the second statement. It suffices to show it for the case that $\psi\ne0$ or $\psi'\ne0$. Let $h\colon (E,\psi)\to(E',\psi')$ be a nontrivial homomorphism and $G:=\im h$. 

Assume that $h\colon (E,\psi)\to(E',\psi')$ is not an isomorphism. If $h$ is not injective, then the induced homomorphism $(E,\psi)\to(G,\psi_{G})$ is not an isomorphism. So $$\dss\frac{P_{(E,\psi)}}{\rk E}<\frac{P_{(G,\psi_{G})}}{\rk G}\le\frac{P_{(G,\psi'_{G})}}{\rk G}\le\frac{P_{(E',\psi')}}{\rk E'}$$ which is a contradiction to the assumption. If $h$ is not surjective, then the induced homomorphism $(G,\psi'_{G})\to(E',\psi')$ are not an isomorphism. So $$\dss\frac{P_{(E,\psi)}}{\rk E}\le\frac{P_{(G,\psi_{G})}}{\rk G}\le\frac{P_{(G,\psi'_{G})}}{\rk G}<\frac{P_{(E',\psi')}}{\rk E'}$$ which is a contradiction to the assumption. If $h$ is an isomorphism and $w=0$, then $\psi', \psi'_{G}$ are zero and $\psi, \psi_{G}$ are not zero. Thus $$\dss\frac{P_{(E,\psi)}}{\rk E}=\frac{P_{(G,\psi_{G})}}{\rk G}<\frac{P_{(G,\psi'_{G})}}{\rk G}=\frac{P_{(E',\psi')}}{\rk E'}$$ which is a contradiction to the assumption.

Finally the proof of the third statement is same as in the proof of \cite[Lemma 1.6]{HL95b}.
\end{proof}
\begin{rem}
The second and third statements of Lemma \ref{vanishing hom} are in \cite[Lemma 1.6]{HL95b}.
\end{rem}

\begin{lem}\label{max dest}
Let $(E,\psi)$ be a framed module. Then there is a framed submodule $(F,\psi_{F})$ such that for all framed submodules $(G,\psi')\subset(E,\psi)$ one has
$$\frac{P_{(F,\psi_{F})}}{\rk F}\ge\frac{P_{(G,\psi')}}{\rk G},$$
and in case of equality $(F,\psi_{F})\supset (G,\psi')$. Moreover, $(F,\psi_{F})$ is uniquely determined and $\s$-semistable.
\end{lem}
\begin{proof}
The proof is the same as the standard argument of the proof of \cite[Lemma 1.3.5]{HL10} by using the reduced Hilbert polynomials, stabilities and an order relation of framed modules instead of the usual ones.
\end{proof}

Such a framed submodule $(F,\psi_{F})$ of $(E,\psi)$ in Lemma \ref{max dest} is called {\bf the maximal destabilizing framed submodule of $(E,\psi)$}.

\begin{prop}\label{HN filt}
Assume that a framed module $(E,\psi)$ is not $\s$-semistable. Then there exists a unique Harder-Narasimhan filtration of framed submodules
$$(0,0)\subset(E_{1},\psi_{E_{1}})\subset\cdots\subset(E_{l-1},\psi_{E_{l-1}})\subset(E_{l},\psi_{E_{l}})=(E,\psi)$$
such that $(E_{i}/E_{i-1},\psi_{E_{i}/E_{i-1}})$ are $\s$-semistable and
$$\frac{P_{(E_{i}/E_{i-1},\psi_{E_{i}/E_{i-1}})}}{\rk E_{i}-\rk E_{i-1}}>\frac{P_{(E_{i+1}/E_{i},\psi_{E_{i+1}/E_{i}})}}{\rk E_{i+1}-\rk E_{i}}\text{ for }i=1,\dots,l-1.$$
\end{prop}
\begin{proof}
The proof is the same as the argument below the proof of \cite[Lemma 1.3.5]{HL10} by using the reduced Hilbert polynomials and stabilities of framed modules instead of the usual ones combined with Lemmas \ref{vanishing hom} and \ref{max dest}.
\end{proof}

\subsection{Framed Hitchin pairs}\label{framed Hitchin pairs}

Let $L$ and $\cH$ be vector bundles on $X$.

{\bf A framed Hitchin pair of type $(r,d,L,\cH)$ on $X$} is a quadruple $(E,\ve,\phi,\psi)$ which consists of the following ingredients
\begin{itemize}
\item a vector bundle $E$ of rank $r$ and degree $d$ on $X$;
\item a complex number $\ve$;
\item a homomorphism $\phi:E\to E\otimes L$, called the $L$-twisted Higgs field of $E$;
\item a framing $\psi:E\to\cH$
\end{itemize}
such that $\phi$ is not nilpotent, i.e., $(\phi\otimes\id_{L^{\otimes i-1}})\circ\cdots\circ\phi\ne0$ for all $i\in\nn$, when $\ve=0$  (see \cite[Section 1]{Sch00}). {\bf A constrained framed Hitchin pair on $X$} is a framed Hitchin pair $(E,\ve,\phi,\psi)$ with $\phi\in\End((E,\psi))$.

Let $(E,\ve,\phi,\psi)$ and $(E',\ve',\phi',\psi')$ be framed Hitchin pairs of type $(r,d,L,\cH)$ and $(r',d',L,\cH')$ respectively. {\bf A homomorphism $h:(E,\ve,\phi,\psi)\to(E',\ve',\phi',\psi')$ of framed Hitchin pairs} is a homomorphism of the underlying bundles $h:E\to E'$ for which there are elements $z,w\in\cc$ such that $\ve'=z\ve,\ \phi'\circ h=z((h\otimes\id_L)\circ \phi),\ \textrm{and}\ \psi'\circ h=w\psi$. The set $\Hom((E,\ve,\phi,\psi),(E',\ve,\phi',\psi'))$ of homomorphisms of framed Hitchin pairs is a linear subspace of $\Hom(E,E')$. If {\bf $h:(E,\ve,\phi,\psi)\to(E',\ve,\phi',\psi')$ is an isomorphism}, then it is an isomorphism of the underlying bundles $h:E\to E'$ and the factors $z,w$ are taken in $\cc^{*}$.

Fix a positive $\s\in\qq$. A framed Hitchin pair $(E,\ve,\phi,\psi)$ is {\bf$\s$-(semi)stable}, if for any proper, nonzero subbundle $F$ of $E$ which is $\phi$-invariant,
$$\frac{P_{(F,\psi_{F})}}{\rk F}<(\le)\frac{P_{(E,\psi)}}{\rk E}.$$

Let $\bFH^{\s-ss}(r,d,L,\cH)$ be the moduli space of $\s$-semistable framed Hitchin pairs of type $(r,d,L,\cH)$ with nonzero framings on $X$. $\bFH^{\s-ss}(r,d,L,\cH)$ was constructed in \cite{Sch00}. $\bFH^{\s-s}(r,d,L,\cH)$ denotes the stable locus of $\bFH^{\s-ss}(r,d,L,\cH)$.

\begin{rem}
The construction of $\bFH^{\s-ss}(r,d,L,\cH)$ in \cite{Sch00} has a generality. The moduli space parametrizes objects with underlying torsion free coherent sheaves on a smooth complex projective scheme, where $\s$ is a rational polynomial with a positive leading coefficient.
\end{rem}

Now we mention that a numerical property of $\s$ holds for any $\s$-semistabie framed Hitchin pair $(E,\ve,\phi,\psi)$ with $\psi\neq 0$ and $\phi$-invariant nonzero $\ker\psi$ as follows.

\begin{prop}\label{numerical property}
Assume that a framed Hitchin pair $(E,\ve,\phi,\psi)$ with $\psi\neq 0$ and $\ker\psi\ne0$ is $\s$-semistable and that $\ker\psi$ is $\phi$-invariant. Then
$$\s\le \deg E-\frac{\rk E}{\rk\ker\psi}(\deg E-\deg\cH).$$
\end{prop}
\begin{proof}
By the assumption, we have $$\dss\frac{\deg\ker\psi}{\rk\ker\psi}\le\frac{\deg E-\s}{\rk E}.$$ Then
\begin{align*}
\s & \le \deg E-\frac{\rk E}{\rk\ker\psi}\deg\ker\psi=\deg E-\frac{\rk E}{\rk\ker\psi}(\deg E-\deg\im\psi)\\
    & \le \deg E-\frac{\rk E}{\rk\ker\psi}(\deg E-\deg\cH).
\end{align*}
\end{proof}

\subsection{Oriented framed modules and oriented framed Hitchin pairs}
Fix $L$ and $\cH$ as in the previous section. Let $\cN$ be a Poincar\'{e} line bundle on $\Pic(X)\times X$.

{\bf An oriented framed module of type $(r,d,\cH,\cN)$ on $X$} is a triple $(E,\delta,\psi)$ which consists of the following ingredients
\begin{itemize}
\item a vector bundle $E$ of rank $r$ and degree $d$ on $X$;
\item a homomorphism $\delta\colon  \det E\to \cN[E]$;
\item a framing $\psi\colon E\to\cH$
\end{itemize}
(see \cite[Section 2.1]{OST99}).

An {\bf oriented framed Hitchin pair of type $(r,d,L,\cH,\cN)$ on $X$} is a quintuple $(E,\ve,\delta, \phi,\psi)$ which consists of the following ingredients
\begin{itemize}
\item a vector bundle $E$ of rank $r$ and degree $d$ on $X$;
\item a complex number $\ve$;
\item a homomorphism $\delta\colon  \det E\to \cN[E]$;
\item a $L$-twisted Higgs field $\phi\colon E\to E\otimes L$;
\item a framing $\psi\colon E\to\cH$
\end{itemize}
such that $\phi$ is not nilpotent when $\ve=0$ (see \cite[Section 2]{Sch00}). {\bf A constrained oriented framed Hitchin pair on $X$} is an oriented framed Hitchin pair $(E,\ve,\delta, \phi,\psi)$ with $\phi\in\End((E,\psi))$.

{\bf An isomorphism $h\colon(E,\ve,\delta, \phi, \psi)\to(E',\ve',\delta', \phi', \psi')$ of oriented framed Hitchin pairs of type $(r,d,L,\cH,\cN)$ on $X$} is an isomorphism of the underlying bundles $h\colon E\to E'$ for which there are elements $z, w\in\cc^{*}$ such that 
\[
\ve'=z\ve,\ \delta'\circ \det h=w^r\delta,\ \phi'\circ h=z((h\otimes\id_L)\circ \phi),\ \textrm{and}\ \psi'\circ h=w\psi.
\]

Denote $\dss\s_{E,\phi,\psi}:=P_E-\frac{\rk E}{\rk K_{\textrm{max}}}P_{K_{\textrm{max}}}$ (see \cite[Lemma 2.2]{Sch00} for $K_{\textrm{max}}$). We denote $\dss\s_{E,\phi,\psi}$ by $\dss\s_{E,\psi}$ when $\phi=0$.

An oriented framed module $(E,\delta,\psi)$ is {\bf semistable} if either $\psi$ is injective, or $\delta$ is an isomorphism, $\ker\psi\ne0$, $\s_{E,\psi}\ge0$ and for any nontrivial subbundle $F\subset E$
\[
\frac{P_F}{\rk F}-\frac{\s_{E,\psi}}{\rk F}\le \frac{P_E}{\rk E}-\frac{\s_{E,\psi}}{\rk E}.
\]

An oriented framed module $(E,\delta,\psi)$ is {\bf stable} if either $\psi$ is injective, or $\delta$ is an isomorphism, $\ker\psi\ne0$, $\s_{E,\psi}>0$ and one of the following two possibilities holds:
\begin{itemize}
	\item For any nontrivial proper subbundle $F\subset E$ $$\frac{P_F}{\rk F}-\frac{\s_{E,\psi}}{\rk F}< \frac{P_E}{\rk E}-\frac{\s_{E,\psi}}{\rk E}.$$
	\item $\psi\neq0$, $(E,\psi)$ splits as $(K_{\textrm{max}},0)\oplus(E',\psi)$ such that $K_{\textrm{max}}$ is stable and $(E',\psi)$ is a $\s_{E,\psi}$-stable framed module with
	$$\dss\frac{P_{K_{\textrm{max}}}}{\rk K_{\textrm{max}}}=\frac{P_{E'}-\s_{E,\psi}}{\rk E'}.$$
\end{itemize}

An oriented framed Hitchin pair $(E,\ve,\delta,\phi,\psi)$ is {\bf semistable} if either there is no $\phi$-invariant subbundles in $\ker \psi$, or $\delta$ is an isomorphism and there are $\phi$-invariant subbundles in $\ker \psi,\ \s_{E,\phi,\psi}\ge0$ and for any nontrivial $\phi$-invariant subbundle $F\subset E$
\[
\frac{P_F}{\rk F}-\frac{\s_{E,\phi,\psi}}{\rk F}\le \frac{P_E}{\rk E}-\frac{\s_{E,\phi,\psi}}{\rk E}.
\]

An oriented framed Hitchin pair $(E,\ve,\delta,\phi,\psi)$ is {\bf stable} if either there is no $\phi$-invariant subbundles in $\ker \psi$, or $\delta$ is an isomorphism and there are $\phi$-invariant subbundles in $\ker \psi,\ \s_{E,\phi,\psi}>0$, and one of the following two possibilities holds:
\begin{itemize}
	\item For any nontrivial $\phi$-invariant proper subbundle $F\subset E$ $$\frac{P_F}{\rk F}-\frac{\s_{E,\phi,\psi}}{\rk F}< \frac{P_E}{\rk E}-\frac{\s_{E,\phi,\psi}}{\rk E}.$$
	\item $\psi\neq0$, $(E,\ve,\phi,\psi)$ splits as $(K_{\textrm{max}},\ve,\phi|_{K_{\textrm{max}}},0)\oplus(E',\ve,\phi|_{E'},\psi)$ such that $(K_{\textrm{max}},\ve,\phi|_{K_{\textrm{max}}})$ is a stable Hitchin pair and $(E',\ve,\phi|_{E'},\psi)$ is a $\s_{E,\phi,\psi}$-stable framed Hitchin pair with
	$$\dss\frac{P_{K_{\textrm{max}}}}{\rk K_{\textrm{max}}}=\frac{P_{E'}-\s_{E,\phi,\psi}}{\rk E'}.$$
\end{itemize}

Let $\bOFH^{ss}(r,d,L,\cH,\cN)$ (resp. $\bOFM^{ss}(r,d,\cH,\cN)$) be the moduli space of semistable oriented framed Hitchin pairs of type $(r,d,L,\cH,\cN)$ (resp. semistable oriented framed modules of type $(r,d,\cH,\cN)$) on $X$. $\bOFM^{ss}(r,d,\cH,\cN)$ was constructed in \cite{OST99} and $\bOFH^{ss}(r,d,L,\cH,\cN)$ was constructed in \cite{Sch00}. The stable locus of $\bOFH^{ss}(r,d,L,\cH,\cN)$ (resp. $\bOFM^{ss}(r,d,\cH,\cN)$) is denoted by $\bOFH^{s}(r,d,L,\cH,\cN)$ (resp. $\bOFM^{s}(r,d,\cH,\cN)$).

\begin{rem}
The construction of $\bOFM^{ss}(r,d,\cH,\cN)$ in \cite{OST99} has a generality. The moduli space parametrizes objects with underlying torsion free coherent sheaves on a smooth complex projective variety. However, the construction of $\bOFH^{ss}(r,d,L,\cH,\cN)$ in \cite{Sch00} has a constraint. The moduli space parametrizes objects with underlying vector bundles on a smooth complex projective curve.
\end{rem}

\section{$\cc^{*}$-flips of moduli spaces}\label{flips}
In this section we deal with birational geometries of moduli spaces of framed modules and those of framed Hitchin pairs. See \cite[Section 2]{Sch00} for details. 

Let $\qq_{+}$ be the set of all positive rational numbers. Assume that $(E,\ve,\delta,\phi,\psi)$ is not stable but semistable. Then either there is a $\phi$-invariant subbundle in $\ker\psi$ which destabilizes $(E, \ve, \phi)$ as a Hitchin pair or there are a $\phi$-invariant subbundle $K$ of $\ker\psi$ and a $\phi$-invariant subbundle $F\not\subset\ker\psi$ with
$$\frac{P_F}{\rk F}-\frac{\s_{K}}{\rk F}=\frac{P_E}{\rk E}-\frac{\s_{K}}{\rk E},$$
where $\dss\s_{K}:=P_E-\frac{\rk E}{\rk K}P_K$. Let $Q_{\dest}$ be the subset of $\qq_{+}$ of such $\s_{K}$. Fix $\s_{\infty}\in\qq_{+}$ for which the conclusion of \cite[Proposition 2.9]{Sch00} holds. Then the set $Q_{\dest}\cap\{\s\in\qq_{+}\,|\,\s\le\s_{\infty}\}$ is finite (cf. \cite[Lemma 2.10]{Sch00}).

Let $\s'_{1}<\cdots<\s'_{t}$ be the rational numbers in $Q_{\dest}\cap\{\s\in\qq_{+}\,|\,\s<\s_{\infty}\}$. These give rise to chambers of wall-crossings $I_{0}:=\{\s\,|\,\s<\s'_{1}\}$, $I_{i}:=\{\s\,|\,\s'_{i}<\s<\s'_{i+1}\}$ for $i=1,\dots,t-1$ and $I_{t}:=\{\s\,|\,\s'_{t}<\s\}$. Pick $\s_{i}\in I_{i}$ for each $i\in\{0,\dots,t\}$ and let $\bFH^{\s_{i}-ss}:=\bFH^{\s_{i}-ss}(r,d,L,\cH)$ (resp. $\bFH^{\s_{i}-s}:=\bFH^{\s_{i}-s}(r,d,L,\cH)$). Then we have a diagram of a chain of variations of moduli spaces
\begin{equation}\label{chain of flips}
{\small\xymatrix{\bFH^{\s_{0}-s}\ar[dr]_{G_{1}^{-}}\ar[d]_{\iota_{\bM}}&&\bFH^{\s_{1}-s}\ar[dl]^{G_{1}^{+}}&\cdots&\bFH^{\s_{t-1}-s}\ar[dr]_{G_{t}^{-}}&&\bFH^{\s_{t}-s}\ar[dl]^{G_{t}^{+}}\\
\bM(r,d,L)&\bFH^{\s'_{1}-ss}&&&&\bFH^{\s'_{t}-ss}&}}
\end{equation}
where $\bM(r,d,L)$ denotes the projective moduli space of semistable Hitchin pairs $(E,\ve,\phi\colon E\to E\otimes L)$ of rank $r$ and degree $d$ on $X$ constructed in \cite{Sch98}, $\iota_{\bM}\colon \bFH^{\s_{0}-s}\to\bM(r,d,L)$ is the forgetful morphism given by $(E,\ve,\phi,\psi)\mapsto(E,\ve, \phi)$, $G_{i}^{-}\colon \bFH^{\s_{i-1}-s}\to\bFH^{\s'_{i}-ss}$ (resp. $G_{i}^{+}\colon \bFH^{\s_{i}-s}\to\bFH^{\s'_{i}-ss}$) is the proper birational morphism changing from $\s_{i-1}$-stability to $\s'_{i}$-stability (resp. from $\s_{i}$-stability to $\s'_{i}$-stability) in the sense of \cite[Theorem 3.3]{Tha96}.

\begin{rem}
It is easy to see that the fiber of $\iota_{\bM}\colon \bFH^{\s_{0}-s}\to\bM(r,d,L)$ over the stable locus $\bM(r,d,L)^{s}$ is $\pp(\Hom(E,\cH)^{\vee})$ by the definition of $\s$-semistability and the condition $\s_{0}<\s'_{1}$.
\end{rem}

Consider the $\cc^{*}$-action $\lambda$ on $\bOFH^{ss}(r,d,L,\cH,\cN)$ given by
\begin{equation}\label{C*-action}
\lambda(z)\cdot[(E,\ve,\delta,\phi,\psi)]:=[(E,\ve,\delta,\phi,z\psi)]=[(E,\ve,z^{-r}\delta,\phi,\psi)].
\end{equation}
As in \cite[Section 1]{OST99}, $\lambda$ can be modified to a family of $\cc^{*}$-actions $\lambda_{e}^{k}$ on $\bOFH^{ss}(r,d,L,\cH,\cN)$ for $e,k\in\zz$ with $0<e\le k$.

\begin{prop}[{\cite[Theorem 2.20]{Sch00}}]\label{cc star flips}
For $e,k\in\zz$ with $0<e\le k$, let $l_{k}^{e}$ be the linearization of the $\cc^{*}$-action $\lambda_{e}^{k}$ on $\bOFH^{ss}(r,d,L,\cH,\cN)$. Then the GIT-quotients $\bOFH^{ss}(r,d,L,\cH,\cN)/\!/_{l_{k}^{e}}\cc^{*}$ run through the moduli spaces $\bM(r,d,L)$ and $\bFH^{\s-ss}$ for $\s\in\qq_{+}$. In particular, the chain of variations of moduli spaces (\ref{chain of flips}) is a chain of $\cc^{*}$-flips in the sense of \cite[Theorem 3.3]{Tha96}.
\end{prop}
By Proposition \ref{cc star flips}, we have
\begin{equation}\label{flips via framed Hitchin pairs}
{\tiny\xymatrix{&\tbFH_{1}^{-}\cong\tbFH_{1}^{+}\ar[dl]_{F_{1}^{-}}\ar[dr]^{F_{1}^{+}}&&&&\tbFH_{t}^{-}\cong\tbFH_{t}^{+}\ar[dl]_{F_{t}^{-}}\ar[dr]^{F_{t}^{+}}&\\
\bFH^{\s_{0}-s}\ar[dr]_{G_{1}^{-}}\ar[d]_{\iota_{\bM}}&&\bFH^{\s_{1}-s}\ar[dl]^{G_{1}^{+}}&\cdots&\bFH^{\s_{t-1}-s}\ar[dr]_{G_{t}^{-}}&&\bFH^{\s_{t}-s}\ar[dl]^{G_{t}^{+}}\\
\bM(r,d,L)&\bFH^{\s'_{1}-ss}&&&&\bFH^{\s'_{t}-ss}&}}
\end{equation}
where $F_{i}^{-}\colon \tbFH_{i}^{-}\to\bFH^{\s_{i-1}-s}$ is the blow-up along $(G_{i}^{-})^{-1}(\bFH^{\s'_{i}-ss}\setminus\bFH^{\s'_{i}-s})$, $F_{i}^{+}\colon \tbFH_{i}^{+}\to\bFH^{\s_{i}-s}$ is the blow-up along $(G_{i}^{+})^{-1}(\bFH^{\s'_{i}-ss}\setminus\bFH^{\s'_{i}-s})$ and $\tbFH_{i}^{-}\cong\tbFH_{i}^{+}$.

\begin{rem}\label{cc star flips via framed modules}
The same argument for $\bFM^{\s-ss}(r,d,\cH)$ and $\bOFM^{\s-ss}(r,d,\cH)$ gives the following diagram of $\cc^{*}$-flips
\begin{equation}\label{flips via framed modules}
{\tiny\xymatrix{&\tbFM_{1}^{-}\cong\tbFM_{1}^{+}\ar[dl]_{f_{1}^{-}}\ar[dr]^{f_{1}^{+}}&&&&\tbFM_{t}^{-}\cong\tbFM_{t}^{+}\ar[dl]_{f_{t}^{-}}\ar[dr]^{f_{t}^{+}}&\\
\bFM^{\s_{0}-s}\ar[dr]_{g_{1}^{-}}\ar[d]_{\iota_{\bU}}&&\bFM^{\s_{1}-s}\ar[dl]^{g_{1}^{+}}&\cdots&\bFM^{\s_{t-1}-s}\ar[dr]_{g_{t}^{-}}&&\bFM^{\s_{t}-s}\ar[dl]^{g_{t}^{+}}\\
\bU(r,d)&\bFM^{\s'_{1}-ss}&&&&\bFM^{\s'_{t}-ss}&}}
\end{equation}
where $\bU(r,d)$ denotes the moduli space of semistable vector bundles of the rank $r$ and degree $d$ on $X$, $\iota_{\bU}\colon \bFM^{\s_{0}-s}\to\bU(r,d)$ is the forgetful morphism given by $(E,\psi)\mapsto E$, $f_{i}^{-}\colon \tbFM_{i}^{-}\to\bFM^{\s_{i-1}-s}$ is the blow-up along $(g_{i}^{-})^{-1}(\bFM^{\s'_{i}-ss}\setminus\bFM^{\s'_{i}-s})$, $f_{i}^{+}\colon \tbFM_{i}^{+}\to\bFM^{\s_{i}-s}$ is the blow-up along $(g_{i}^{+})^{-1}(\bFM^{\s'_{i}-ss}\setminus\bFM^{\s'_{i}-s})$ and $\tbFM_{i}^{-}\cong\tbFM_{i}^{+}$. See \cite[Section 2.8]{OST99} for details. Here it is easy to see that the fiber of $\iota_{\bU}\colon \bFM^{\s_{0}-s}\to\bU(r,d)$ over the stable locus $\bU(r,d)^{s}$ is $\pp(\Hom(E,\cH)^{\vee})$ by the definition of $\s$-semistability and the condition $\s_{0}<\s'_{1}$.
\end{rem}

\section{Stabilities of $\cO_{X}$-twisted (oriented) framed Hitchin pairs}\label{stabilities}

Assume that $L=\cO_{X}$. In this section we relate stabilities between constrained (oriented) framed Hitchin pairs and (oriented) framed modules. We obtain the equivalence of $\s$-semistable ones for all cases of ranks. But for $\s$-stable ones, we give their equivalence only for the rank $2$ case.

From now on we may assume that the type of framed Hitchin pairs is $(r,d,\cO_{X},\cH)$ and $\dss\frac{d}{r}<\deg\cH$ by Proposition \ref{numerical property}.

\begin{thm}\label{ss FH equiv ss FM}
A framed Hitchin pair $(E,\ve,\phi,\psi)$ is $\s$-semistable with $\phi\in\End((E,\psi))$ if and only if a framed module $(E,\psi)$ is $\s$-semistable.
\end{thm}
\begin{proof}
If $(E,\psi)$ is $\s$-semistable, then $(E,\ve,\phi,\psi)$ is also $\s$-semistable.

Assume that $(E,\psi)$ is not $\s$-semistable. By Proposition \ref{HN filt}, we can take its Harder-Narasimhan filtration
$$(0,0)\subset(E_{1},\psi_{E_{1}})\subset\cdots\subset(E_{l-1},\psi_{E_{l-1}})\subset(E_{l},\psi_{E_{l}})=(E,\psi)$$
such that $(E_{i}/E_{i-1},\psi_{E_{i}/E_{i-1}})$ are $\s$-semistable and
$$\frac{P_{(E_{i}/E_{i-1},\psi_{E_{i}/E_{i-1}})}}{\rk E_{i}-\rk E_{i-1}}>\frac{P_{(E_{i+1}/E_{i},\psi_{E_{i+1}/E_{i}})}}{\rk E_{i+1}-\rk E_{i}}\text{ for }i=1,\dots,l-1.$$
By Lemma \ref{vanishing hom}, $\Hom((E_{i}/E_{i-1},\psi_{E_{i}/E_{i-1}}),(E_{j}/E_{j-1},\psi_{E_{j}/E_{j-1}}))=0$ for $i<j$.

The framed submodule $(E_{1},\psi_{E_{1}})$ satisfies
$$\frac{P_{(E_{1},\psi_{E_{1}})}}{\rk E_{1}}>\frac{P_{(E,\psi)}}{\rk E}.$$
Consider $\phi\in\End((E,\psi))$. Suppose that $\phi(E_{1})$ is nonzero. Then $\phi$ induces a nonzero homomorphism in
$$\Hom((E_{1},\psi_{E_{1}}),(E_{m}/E_{m-1},\psi_{E_{m}/E_{m-1}}))$$ for some $m\ge 1$,
but there are no nonzero homomorphisms unless $m=1$. Thus $E_{1}$ is $\phi$-invariant, and $(E,\ve,\phi,\psi)$ is not $\s$-semistable.
\end{proof}

\begin{prop}[{cf. \cite[Proposition 2.9]{Sch00}}]\label{FM in the final chamber}
A framed module $(E,\psi)$ with $\psi\ne0$ is $\s$-stable for all sufficiently large positive $\s\in\qq$ if and only if $\ker\psi=0$.
\end{prop}
\begin{proof}
(`Only if' part) Assume that $(E,\psi)$ is $\s$-stable for all sufficiently large positive $\s\in\qq$ and $\ker\psi\ne0$. Then
$$\frac{\deg \ker\psi}{\rk\ker\psi}<\frac{\deg E}{\rk E}-\frac{\s}{\rk E}.$$
Then $\s$ is bounded above, which is a contradiction.

(`If' part) Assume that $\ker\psi=0$. Let $G$ be a nonzero proper subbundle of $E$. Since $G\cong\psi(G)$, $\displaystyle\frac{\deg G}{\rk G}\le\frac{\deg \cH_{1}}{\rk\cH_{1}}$, where the Harder-Narasimhan filtration of $\cH$ is $0\subset\cH_{1}\subset\cdots\subset\cH_{l-1}\subset\cH_{l}$. Then $\displaystyle\frac{\deg G}{\rk G}<\frac{\deg E}{\rk E}+\left(\frac{1}{\rk G}-\frac{1}{\rk E}\right)\s$ for all sufficiently large positive $\s\in\qq$.
\end{proof}

\begin{thm}\label{st FH equiv st FM}
\begin{enumerate}
\item For a semistable vector bundle $E$ on $X$, a framed Hitchin pairs $(E,\ve,\phi,\psi)$ is $\s$-stable with $\phi\in\End((E,\psi))$ and $\psi\ne0$ if and only if a framed module $(E,\psi)$ with $\psi\ne0$ is $\s$-stable.
\item Let $E$ be an unstable vector bundle on $X$ with the Harder-Narasimhan filtration
$$0=E_{0}\subset E_{1}\subset\cdots\subset E_{n}=E$$
satisfying $E_{1}\subset\ker\psi$. Then a framed Hitchin pairs $(E,\ve,\phi,\psi)$ is $\s$-stable with $\phi\in\End((E,\psi))$ and $\psi\ne0$ if and only if a framed module $(E,\psi)$ with $\psi\ne0$ is $\s$-stable.
\end{enumerate}
\end{thm}
\begin{proof}
It is obvious that if a framed module $(E,\psi)$ is $\s$-stable, then a framed Hitchin pairs $(E,\ve,\phi,\psi)$ is $\s$-stable. So, it suffices to prove `only if' parts.

Assume that $(E,\ve,\phi,\psi)$ is $\s$-stable with $\phi\in\End((E,\psi))$ and $(E,\psi)$ is not $\s$-stable. By Theorem \ref{ss FH equiv ss FM}, $(E,\psi)$ is not $\s$-stable but $\s$-semistable. Then there is the maximal destabilizing submodule $(K_{\max},0)$ with $K_{\max}\subset\ker\psi$ and
\begin{equation}\label{K max}\frac{P_{K_{\max}}}{\rk K_{\max}}=\frac{P_{E}}{\rk E}-\frac{\s}{\rk E}\end{equation}
or the maximal destabilizing submodule $(F_{\max},\psi_{F_{\max}})$ with $F_{\max}\not\subset\ker\psi$ and
\begin{equation}\label{F max}\frac{P_{F_{\max}}}{\rk F_{\max}}-\frac{\s}{\rk F_{\max}}=\frac{P_{E}}{\rk E}-\frac{\s}{\rk E}.\end{equation}
But we see that $\ker\psi$ is $\phi$-invariant because $\phi\in\End((E,\psi))$. Note that $\ker\psi\ne E$ because $\psi\ne0$.

We first claim that there is no such $K_{\max}$. We have only to show for the case $\ker\psi\ne0$. Since $(E,\ve,\phi,\psi)$ is $\s$-stable, we have
$$\frac{P_{\ker\psi}}{\rk\ker\psi}<\frac{P_{E}}{\rk E}-\frac{\s}{\rk E}.$$
By (\ref{K max}), we get $\dss\frac{P_{K_{\max}}}{\rk K_{\max}}>\frac{P_{\ker\psi}}{\rk\ker\psi}$. Thus $K_{\max}\subsetneq\ker\psi$, and so $\ker\psi$ is unstable. Now we can take the Harder-Narasimhan filtration of $\ker\psi$
$$0=G_{0}\subset G_{1}\subset\cdots\subset G_{m}=\ker\psi$$
such that $G_{i}/G_{i-1}$ are semistable and
$$\frac{P_{G_{i}/G_{i-1}}}{\rk(G_{i}/G_{i-1})}>\frac{P_{G_{i+1}/G_{i}}}{\rk(G_{i+1}/G_{i})}\text{ for }i=1,\dots,m-1.$$
Then we obtain that $G_1$ is $\phi$-invariant. Indeed, suppose not, then $\phi$ induces a nonzero homomorphism $\bar{\phi}:G_{1}\to G_{i}/G_{i-1}$ for some $i>1$. But since $\dss\frac{P_{G_{1}}}{\rk G_{1}}>\frac{P_{G_{i}/G_{i-1}}}{\rk(G_{i}/G_{i-1})}$, $\bar{\phi}$ must be zero by  \cite[Proposition 1.2.7]{HL10}, which is a contradiction.

By \cite[Lemma 1.3.5]{HL10}, we have $\dss\frac{P_{G_{1}}}{\rk G_{1}}\ge\frac{P_{K_{\max}}}{\rk K_{\max}}$, and $G_{1}\supset K_{\max}$ when the equality holds. It violates the $\s$-stability of $(E,\ve,\phi,\psi)$ and (\ref{K max}) since $G_1$ is $\phi$-invariant. Hence we do not have such $K_{\max}$.
\begin{enumerate}
\item Assume that $E$ is semistable. Then we get a contradiction because $\dss\frac{P_{F_{\max}}}{\rk F_{\max}}-\frac{P_{E}}{\rk E}=\frac{\s}{\rk F_{\max}}-\frac{\s}{\rk E}>0$.
\item Assume that $E$ is unstable with the Harder-Narasimhan filtration
$$0=E_{0}\subset E_{1}\subset\cdots\subset E_{n}=E$$
such that $E_{1}\subset\ker\psi$, $E_{i}/E_{i-1}$ are semistable and
$$\frac{P_{E_{i}/E_{i-1}}}{\rk(E_{i}/E_{i-1})}>\frac{P_{E_{i+1}/E_{i}}}{\rk(E_{i+1}/E_{i})}\text{ for }i=1,\dots,n-1.$$
Note that $\dss\frac{P_{E_{1}}}{\rk E_{1}}\ge\frac{P_{F_{\max}}}{\rk F_{\max}}$, and $E_{1}\supset F_{\max}$ when the equality holds (see \cite[Lemma 1.3.5]{HL10}).

If $\dss\frac{P_{E_{1}}}{\rk E_{1}}=\frac{P_{F_{\max}}}{\rk F_{\max}}$, then $E_{1}\supset F_{\max}$, and so $E_{1}\not\subset\ker\psi$ which is a contradiction. Thus we get $\dss\frac{P_{E_{1}}}{\rk E_{1}}>\frac{P_{F_{\max}}}{\rk F_{\max}}$. On the other hand, since $E_{1}\subset\ker\psi$ and $E_{1}$ is $\phi$-invariant by a similar argument of the above $\phi$-invariance of $G_1$, we have $\dss\frac{P_{E_{1}}}{\rk E_{1}}<\frac{P_{F_{\max}}}{\rk F_{\max}}-\frac{\s}{\rk F_{\max}}$ by the $\s$-stability of $(E,\ve,\phi,\psi)$ and (\ref{F max}). Therefore  $\dss\frac{P_{E_{1}}}{\rk E_{1}}<\frac{P_{F_{\max}}}{\rk F_{\max}}$, which is a contradiction.
\end{enumerate}
\end{proof}

\begin{cor}\label{st FH equiv st FM in rank 2}
For a vector bundle $E$ of rank $2$ on $X$, a framed Hitchin pairs $(E,\ve,\phi,\psi)$ is $\s$-stable with $\phi\in\End((E,\psi))$ and $\psi\ne0$ if and only if a framed module $(E,\psi)$ with $\psi\ne0$ is $\s$-stable.
\end{cor}
\begin{proof}
It suffices to prove `only if' part. By Theorem \ref{st FH equiv st FM}, we assume that $E$ is an unstable vector bundle of rank $2$ on $X$ with the Harder-Narasimhan filtration $0=E_{0}\subsetneq E_{1}\subsetneq E_{2}=E$ with $E_{1}\not\subset\ker\psi$. 

We follow the same argument in the proof of Theorem \ref{st FH equiv st FM}. Then we have either $F_{\max}=E_{1}$ or $F_{\max}\ne E_{1}$.
\begin{itemize}
\item $F_{\max}=E_{1}$ : Since $E_{1}$ is $\phi$-invariant, $\displaystyle P_{F_{\max}}<\frac{P_E}{2}+\frac{\s}{2}=P_{F_{\max}}$, which is impossible.
\item $F_{\max}\ne E_{1}$ : Since $E_{1}$ is the maximal destabilizing subbundle of $E$ and $\phi$-invariant, $\displaystyle P_{F_{\max}}\le P_{E_{1}}<\frac{P_E}{2}+\frac{\s}{2}=P_{F_{\max}}$, which is impossible.
\end{itemize}
\end{proof}

\begin{rem}
We are not sure that there exists an example that $(E,\ve,\phi,\psi)$ is $\s$-stable and the underlying framed module $(E,\psi)$ is not $\s$-stable such that $\rk E\ge 3$ and $E_{1}\not\subset \ker\psi$.
\end{rem}

We have similar results of $\s$-stabilities for oriented cases.
\begin{thm}\label{ss OFH equiv ss OFM}
An oriented framed Hitchin pair $(E,\ve,\delta,\phi,\psi)$ is semistable with $\phi\in\End((E,\psi))$ if and only if an oriented framed module $(E,\delta,\psi)$ is semistable.
\end{thm}
\begin{proof}
It is an immediate consequence from \cite[Lemma 2.1.1]{OST99}, \cite[Lemma 2.3, Corollary 2.8 and Proposition 2.9]{Sch00}, Theorem \ref{ss FH equiv ss FM} and Proposition \ref{FM in the final chamber}.
\end{proof}

\begin{thm}\label{st OFH equiv st OFM in rank 2}
For a vector bundle $E$ with the rank $2$ on $X$, an oriented framed Hitchin pair $(E,\ve,\delta,\phi,\psi)$ is stable with $\phi\in\End((E,\psi))$ and $\psi\ne0$ if and only if an oriented framed module $(E,\delta,\psi)$ with $\psi\ne0$ is stable.
\end{thm}
\begin{proof}
It is an immediate consequence from \cite[Lemma 2.1.1]{OST99}, \cite[Lemma 2.3, Corollary 2.8 and Proposition 2.9]{Sch00}, Proposition \ref{FM in the final chamber} and Corollary \ref{st FH equiv st FM in rank 2}.
\end{proof}

\section{$\cc^{*}$-flips of moduli spaces of rank 2 objects on a smooth projective curve}\label{cstar flips}
In this section we connect the diagram (\ref{flips via framed Hitchin pairs}) of Section \ref{flips} to the one (\ref{flips via framed modules}) under some constraint of Higgs fields by using forgetful morphisms.

We use the following denotations:
\begin{enumerate}
\item[] $\bOFH:=\bOFH^{ss}(2,d,\cO_{X},\cH,\cN)$;
\item[] $\bOFH_{\fr\ne0}:=$ the locus of oriented framed Hitchin pairs $(E,\ve,\delta,\phi,\psi)$ of $\bOFH$ with $\psi\ne0$;
\item[] $\bOFH^{s}:=\bOFH\cap\bOFH^{s}(2,d,\cO_{X},\cH,\cN)$;
\item[] $\bOFH_{\fr\ne0}^{s}:=\bOFH_{\fr\ne0}\cap\bOFH^{s}(2,d,\cO_{X},\cH,\cN)$;
\item[] $\bOFM:=\bOFM^{ss}(2,d,\cO_{X},\cH,\cN)$;
\item[] $\bOFM_{\fr\ne0}:=$ the locus of oriented framed modules $(E,\delta,\psi)$ of $\bOFM$ with $\psi\ne0$;
\item[] $\bOFM^{s}:=\bOFM\cap\bOFM^{s}(2,d,\cO_{X},\cH,\cN)$;
\item[] $\bOFM_{\fr\ne0}^{s}:=\bOFM_{\fr\ne0}\cap\bOFM^{s}(2,d,\cO_{X},\cH,\cN)$;
\item[] $\bOFH_{\con}:=$ the locus of oriented framed Hitchin pairs $(E,\ve,\delta,\phi,\psi)$ of $\bOFH$ with $\phi\in\End((E,\psi))$;
\item[] $\bOFH_{\con,\fr\ne0}:=\bOFH_{\con}\cap\bOFH_{\fr\ne0}$;
\item[] $\bOFH_{\con}^{s}:=\bOFH_{\con}\cap\bOFH^{s}$;
\item[] $\bOFH_{\con,\fr\ne0}^{s}:=\bOFH_{\con,\fr\ne0}\cap\bOFH^{s}$;
\item[] $\bFH_{\con}^{\s-ss}:=$ the locus of framed Hitchin pairs $(E,\ve,\phi,\psi)$ of $\bFH^{\s-ss}(2,d,\cO_{X},\cH)$ with $\phi\in\End((E,\psi))$;
\item[] $\bFH_{\con}^{\s-s}:=\bFH_{\con}^{\s-ss}\cap\bFH^{\s-s}(2,d,\cO_{X},\cH)$;
\item[] $\tbFH_{i,\con}^{-}:=(F_{i}^{-})^{-1}(\bFH_{\con}^{\s_{i-1}-s})$;
\item[] $\tbFH_{i,\con}^{+}:=(F_{i}^{+})^{-1}(\bFH_{\con}^{\s_{i}-s})$;
\item[] $\bM_{\con}:=\iota_{\bM}(\bFH_{\con}^{\s_{0}-s})$.
\end{enumerate}

\begin{rem}
We note that $\bFM^{\s-ss}(2,d,\cH)$ and $\bFH^{\s-ss}(2,d,\cO_{X},\cH)$  have nonzero framings (see Sections \ref{framed modules} and \ref{framed Hitchin pairs}).
\end{rem}

\begin{thm}\label{commutativity of full diagram}
\begin{enumerate}
\item There exists a surjective forgetful morphism $p_{\s}\colon \bFH_{\con}^{\s-ss}\to\bFM^{\s-ss}$ given by $(E,\ve,\phi,\psi)\mapsto(E,\psi)$. Further $p_{\s}$ is a $\pp^{1}$-bundle over $\bFM^{\s-s}$.
\item\label{commutativity of full diagram1} $p_{\s'_{i}}\circ G_{i}^{-}=g_{i}^{-}\circ p_{\s_{i-1}}$ for each $i=1,\dots,t$.
\item\label{commutativity of full diagram2} $p_{\s'_{i}}\circ G_{i}^{+}=g_{i}^{+}\circ p_{\s_{i}}$ for each $i=1,\dots,t$.
\item There exists a surjective morphism $\tp_{i}\colon \tbFH_{i,\con}^{\pm}\to\tbFM_{i}^{\pm}$ such that $p_{\s_{i-1}}\circ F_{i}^{-}=f_{i}^{-}\circ\tp_{i}$ and $p_{\s_{i}}\circ F_{i}^{+}=f_{i}^{+}\circ\tp_{i}$ for each $i=1,\dots,t$.
\end{enumerate}
\end{thm}
\begin{proof}
\begin{enumerate}
\item The first statement is an immediate consequence from Theorem \ref{ss FH equiv ss FM}.

Let's prove the second statement. Recall the GIT constructions of $\bFM^{\s-ss}$ and $\bFH^{\s-ss}$ from \cite{HL95b} and \cite{Sch00} respectively. Assume that $d$ and $\deg\cH$ are sufficiently large so that \cite[Theorem 2.1 and Proposition 3.2]{HL95b} and \cite[Proposition 1.5, Assumptions 1.12 and 1.13]{Sch00} hold. Let $V$ be a complex vector space of dimension $p:=d+2(1-g)$. Let $\fQ$ be the quot scheme of quotients $q:V\otimes\cO_{X}\to E$ of $V\otimes\cO_{X}$ with rank $2$ and degree $d$ such that $H^{0}(q):V\to H^{0}(E)$ is an isomorphism. Set $\bR:=\pp(\Hom(V,H^{0}(\cH))^{\vee})$.

Let $Z\subset\fQ\times\bR$ be the closed subscheme of points
$$([q:V\otimes\cO_{X}\to E],[a:V\to H^{0}(\cH)])$$
for which $a:V\otimes\cO_{X}\to\cH$ factors through $q$ and induces a framing $\psi:E\to\cH$. $\SL(V)$ acts on $Z$ as follows:
\begin{center}$g\cdot([q],[a])=([q\circ(g\otimes\id_{\cO_{X}})],[a\circ g])$ for $g\in\SL(V)$.\end{center}
Then $\bFM^{\s-ss}=Z^{\s-ss}/\!/_{\cL_{Z}}\SL(V)$ (resp. $\bFM^{\s-s}=Z^{\s-s}/\!/_{\cL_{Z}}\SL(V)$) for some linearization $\cL_{Z}$ on $Z$, where $Z^{\s-ss}$ (resp. $Z^{\s-s}$) denotes the open subscheme of semistable (resp. stable) points with respect to the linearization $\cL_{Z}$ such that the corresponding framed module is $\s$-semistable (resp. $\s$-stable).

Consider the universal quotient $q_{\fQ}:V\otimes\cO_{\fQ\times X}\to\fE_{\fQ}$ on $\fQ\times X$. Define
$$\widehat{\fP}:=\pp(\cO_{\fQ}\oplus\pi_{\fQ*}\EEnd\fE_{\fQ}).$$

There is a tautological line bundle $\fN_{\widehat{\fP}}$ on $\widehat{\fP}$ such that the tautological surjection gives us a homomorphism on $\widehat{\fP}\times X$
$$V\otimes\pi_{X}^{*}\cO_{X}\to\fE_{\widehat{\fP}}\otimes\pi_{\widehat{\fP}}^{*}\fN_{\widehat{\fP}},$$
where $q_{\widehat{\fP}}:V\otimes\cO_{\widehat{\fP}\times X}\to\fE_{\widehat{\fP}}$ is the pullback of the universal quotient. Define $\fP$ as the closed subscheme of $\widehat{\fP}$ where this homomorphism factors through $\fE_{\widehat{\fP}}$.

Let $\fR\subset\fP\times\bR$ be the closed subscheme of points
$$([q:V\otimes\cO_{X}\to E],[\ve\in\cc,\phi:E\to E],[a:V\to H^{0}(\cH)])$$
for which $a:V\otimes\cO_{X}\to\cH$ factors through $q$ and induces a framing $\psi_{q,a}:E\to\cH$. $\SL(V)$ acts on $\fR$ as follows:
\begin{center}$g\cdot([q],[\ve,\phi],[a])=([q\circ(g\otimes\id_{\cO_{X}})],[\ve,g^{-1}\circ\phi\circ g],[a\circ g])$ for $g\in\SL(V)$.\end{center}
Then $\bFH^{\s-ss}=\fR^{\s-ss}/\!/_{\cL_{\fR}}\SL(V)$ (resp. $\bFH^{\s-s}=\fR^{\s-s}/\!/_{\cL_{\fR}}\SL(V)$) for some linearization $\cL_{\fR}$ on $\fR$, where $\fR^{\s-ss}$ (resp. $\fR^{\s-s}$) denotes the open subscheme of semistable (resp. stable) points with respect to the linearization $\cL_{\fR}$ such that the corresponding framed Hitchin pair is $\s$-semistable (resp. $\s$-stable).

Let $\widetilde{\pi}:\fP\to\fQ$ be the natural projection, and $\pi:\fR\to Z$ be the restriction of $\widetilde{\pi}\times\id_{\bR}:\fP\times\bR\to\fQ\times\bR$. Let $\fR_{\con}:=\{([q],[\ve,\phi],[a])\in\fR\,|\,\psi_{q,a}\circ\phi=w \psi_{q,a}\text{ for some }w\in\cc\}$. Denote $\fR_{\con}^{\s-ss}:=\fR_{\con}\cap\fR^{\s-ss}$ and $\fR_{\con}^{\s-s}:=\fR_{\con}\cap\fR^{\s-s}$. Then $\pi|_{\fR_{\con}^{\s-ss}}:\fR_{\con}^{\s-ss}\to Z^{\s-ss}$ is surjective by Theorem \ref{ss FH equiv ss FM} and $\pi|_{\fR_{\con}^{\s-s}}:\fR_{\con}^{\s-s}\to Z^{\s-s}$ is surjective by Theorem \ref{st FH equiv st FM in rank 2}. Further since $\pi|_{\fR_{\con}^{\s-ss}}$ is $\SL(V)$-equivariant, $\pi|_{\fR_{\con}^{\s-ss}}$ induces $p_{\s}$ after taking quotients by $\SL(V)$ on both sides. Let $\{U_{\alpha}\}_{\alpha}$ be an open cover of $\fQ$ so that $\fP|_{U_{\alpha}}$ is trivial. Then it is easy to see that for $U_{\alpha}$ containing $[q:V\otimes\cO_{X}\to E]\in\fQ$
$$(\widetilde{\pi}\times\id_{\bR})^{-1}((U_{\alpha}\times\bR)\cap Z^{\s-s})\cap\fR_{\con}^{\s-s}\cong((U_{\alpha}\times\bR)\cap Z^{\s-s})\times\pp^{1}$$
because $(\widetilde{\pi}\times\id_{\bR})^{-1}((E,\psi))\cap\fR_{\con}^{\s-s}=\pp(\cc\oplus\End((E,\psi)))\cong\pp^{1}$ for any $\s$-stable framed module $(E,\psi)$ by Lemma \ref{vanishing hom}. Since $\SL(V)$ acts on $\pp^{1}$ by a conjugation, it acts trivially. Hence
\begin{align*}
[(\widetilde{\pi}\times\id_{\bR})^{-1}((U_{\alpha}\times\bR)\cap Z^{\s-s})\cap\fR_{\con}^{\s-s}]/\!/\SL(V) & \cong[((U_{\alpha}\times\bR)\cap Z^{\s-s})\times\pp^{1}]/\!/\SL(V)\\
			&=((U_{\alpha}\times\bR)\cap Z^{\s-s})/\!/\SL(V)\times\pp^{1}.
\end{align*}

\item Since $G_{i}^{-}$ and $g_{i}^{-}$ are induced from the $\cc^{*}$-equivariant inclusions $\bOFH_{\con}^{ss}(-)\subset\bOFH_{\con}^{ss}(0)$ and $\bOFM^{ss}(-)\subset\bOFM^{ss}(0)$ respectively in the sense of \cite[Lemma 3.2]{Tha96} and we have the following commutative diagram with forgetful vertical morphisms
$$\xymatrix{\bOFH_{\con,\fr\ne0}^{ss}(-)\ar@{^{(}->}[r]\ar[d]&\bOFH_{\con,\fr\ne0}^{ss}(0)\ar[d]\\
\bOFM_{\fr\ne0}^{ss}(-)\ar@{^{(}->}[r]&\bOFM_{\fr\ne0}^{ss}(0),}$$
it is an immediate consequence.
\item Since $G_{i}^{+}$ and $g_{i}^{+}$ are induced from the $\cc^{*}$-equivariant inclusions $\bOFH_{\con}^{ss}(+)\subset\bOFH_{\con}^{ss}(0)$ and $\bOFM^{ss}(+)\subset\bOFM^{ss}(0)$ respectively in the sense of \cite[Lemma 3.2]{Tha96} and we have the following commutative diagram with forgetful vertical morphisms
$$\xymatrix{\bOFH_{\con,\fr\ne0}^{ss}(+)\ar@{^{(}->}[r]\ar[d]&\bOFH_{\con,\fr\ne0}^{ss}(0)\ar[d]\\
\bOFM_{\fr\ne0}^{ss}(+)\ar@{^{(}->}[r]&\bOFM_{\fr\ne0}^{ss}(0),}$$
it is an immediate consequence.
\item By Theorems \ref{ss FH equiv ss FM} and \ref{st FH equiv st FM in rank 2}, we have $$\bFH_{\con}^{\s'_{i}-ss}\setminus\bFH_{\con}^{\s'_{i}-s}=(p_{\s'_{i}})^{-1}(\bFM^{\s'_{i}-ss}\setminus\bFM^{\s'_{i}-s}).$$ Then $(p_{\s_{i-1}})^{-1}[(g_{i}^{-})^{-1}(\bFM^{\s'_{i}-ss}\setminus\bFM^{\s'_{i}-s})]=(G_{i}^{-})^{-1}(\bFH_{\con}^{\s'_{i}-ss}\setminus\bFH_{\con}^{\s'_{i}-s})$ by Theorem \ref{commutativity of full diagram}-(\ref{commutativity of full diagram1}) and $(p_{\s_{i}})^{-1}[(g_{i}^{+})^{-1}(\bFM^{\s'_{i}-ss}\setminus\bFM^{\s'_{i}-s})]=(G_{i}^{+})^{-1}(\bFH_{\con}^{\s'_{i}-ss}\setminus\bFH_{\con}^{\s'_{i}-s})$ by Theorem \ref{commutativity of full diagram}-(\ref{commutativity of full diagram2}). Hence we get the result by \cite[Corollary 7.15 of Chapter II]{Hart77}
\end{enumerate}
\end{proof}

\begin{lem}\label{nonzero framing GIT construction}
For $e,k\in\zz$ with $0<e\le k$, $\bOFH_{\fr\ne0}/\!/_{l_{k}^{e}}\cc^{*}\cong\bFH^{\s_{k}^{e}-ss}$ and $\bOFH_{\con,\fr\ne0}/\!/_{l_{k}^{e}}\cc^{*}\cong\bFH_{\con}^{\s_{k}^{e}-ss}$, where $l_{k}^{e}$ is the linearization of the $\cc^{*}$-action $\lambda_{e}^{k}$ on $\bOFH$ and $\s_{k}^{e}:=\dss\frac{d+2(1-g)}{2}\cdot\frac{e}{k}$ as in \cite[Section 2]{Sch00}.
\end{lem}
\begin{proof}
From the construction of \cite[Section 2]{Sch00}, $\bOFH_{\fr\ne0}\cong(\fT_{0}\cap\bt^{-1}(\bP_{1}\times\bP_{2,\fr\ne0}))/\!/\SL(V)$, where
$$\bP_{2,\fr\ne0}=\pp\left(\left(\Hom\left(\bigwedge^{2}V\otimes\cO_{\cJ},\fU_{\cJ}\right)\oplus(\rS^{2}\Hom(V,H^{0}(\cH))\setminus\{0\})\right)^{\vee}\otimes\cO_{\cJ}\right).$$
Since the set of semistable points of $\bP_{2}$ with respect to the linearization $l_{k}^{e}$ is nothing but
$$\bP_{2}\setminus\left(\pp\left(\Hom\left(\bigwedge^{2}V\otimes\cO_{\cJ},\fU_{\cJ}\right)^{\vee}\right)\cup\pp((\rS^{2}\Hom(V,H^{0}(\cH)))^{\vee})\right),$$
we have
{\small
$$\bP_{2,\fr\ne0}/\!/_{l_{k}^{e}}\cc^{*}=\bP_{2}/\!/_{l_{k}^{e}}\cc^{*}\cong\begin{cases}\pp\left(\Hom\left(\bigwedge^{2}V\otimes\cO_{\cJ},\fU_{\cJ}\right)^{\vee}\right)\times\pp((\rS^{2}\Hom(V,H^{0}(\cH)))^{\vee})&\text{if }0<e<k\\ \pp((\rS^{2}\Hom(V,H^{0}(\cH)))^{\vee})&\text{if }e=k.\end{cases}$$
}
Thus
\begin{align*}
\bOFH_{\fr\ne0}/\!/_{l_{k}^{e}}\cc^{*} & \cong(\fT_{0}\cap\bt^{-1}(\mathbf{P}_{1}\times\mathbf{P}_{2,\fr\ne0}))/\!/_{\l_{k}^{e}}\SL(V)\times\cc^{*} \\
                                                          & \cong(\fT_{0}\cap\bt^{-1}(\mathbf{P}_{1}\times\mathbf{P}_{2}))/\!/_{\l_{k}^{e}}\SL(V)\times\cc^{*} \\
                                                          & \cong\bOFH/\!/_{l_{k}^{e}}\cc^{*} \\
						       & \cong\bFH^{\s_{k}^{e}-ss}
\end{align*} 
for $0<e\le k$ by \cite[Theorem 2.15 ii) and iii)]{Sch00}.

By the similar way, we also have $\bOFH_{\con,\fr\ne0}/\!/_{l_{k}^{e}}\cc^{*}\cong\bFH_{\con}^{\s_{k}^{e}-ss}$.
\end{proof}

\begin{thm}\label{commutativity of C star filps}
There exists a surjective forgetful morphism $p_{\bOFH}\colon \bOFH_{\con}\to\bOFM$ such that
$$p_{\bOFH}^{-1}(\bOFM_{\fr\ne0}\setminus\bOFM_{\fr\ne0}^{s})=\bOFH_{\con,\fr\ne0}\setminus\bOFH_{\con,\fr\ne0}^{s}.$$
For $e,k\in\zz$ with $0<e\le k$, the GIT-quotient $\bOFH_{\con,\fr\ne0}/\!/_{l_{k}^{e}}\cc^{*}$ run through the moduli spaces $\bFH_{\con}^{\s-ss}$ for $\s\in\qq_{+}$. Further, the $\cc^{*}$-flips from $\bOFH_{\con,\fr\ne0}/\!/_{l_{k}^{e}}\cc^{*}$ commutes with the $\cc^{*}$-flips from $\bOFM_{\fr\ne0}/\!/_{l_{k}^{e}}\cc^{*}$ via the surjective forgetful morphism $\bOFH_{\con,\fr\ne0}/\!/_{l_{k}^{e}}\cc^{*}\to\bOFM_{\fr\ne0}/\!/_{l_{k}^{e}}\cc^{*}$ induced from $p_{\bOFH}$.
\end{thm}
\begin{proof}
By Theorems \ref{ss OFH equiv ss OFM} and \ref{st OFH equiv st OFM in rank 2}, the first statement holds.

Note that $\bOFH_{\con,\fr\ne0}$ (resp. $\bOFM_{\fr\ne0}$) is invariant under the $\cc^{*}$-action on $\bOFH$ (resp. $\bOFM$) (see (\ref{C*-action})). Let $q_{\bOFH}\colon \bOFH_{\con,\fr\ne0}\to\bOFH_{\con,\fr\ne0}/\!/_{l_{k}^{e}}\cc^{*}$ and $q_{\bOFM}\colon \bOFM_{\fr\ne0}\to\bOFM_{\fr\ne0}/\!/_{l_{k}^{e}}\cc^{*}$ be the morphisms of GIT quotients (cf. \cite[Section 2.8]{OST99}). Since $0<e\le k$, we obtain $\bOFH_{\con,\fr\ne0}/\!/_{l_{k}^{e}}\cc^{*}\cong\bFH_{\con}^{\s_{k}^{e}-ss}$ by Lemma \ref{nonzero framing GIT construction}. Thus
\begin{equation}\label{inverse image of qOFH}q_{\bOFH}^{-1}(\bFH_{\con}^{\s_{k}^{e}-ss}\setminus\bFH_{\con}^{\s_{k}^{e}-s})=\bOFH_{\con,\fr\ne0}\setminus\bOFH_{\con,\fr\ne0}^{s},\end{equation}
which implies that the second statement holds by Proposition \ref{cc star flips}.

Since $0<e\le k$, we also have $\bOFM_{\fr\ne0}/\!/_{l_{k}^{e}}\cc^{*}\cong\bFM^{\s_{k}^{e}-ss}$ by \cite[Section 2.5]{OST99} and the similar way as the proof of Lemma \ref{nonzero framing GIT construction}. Then
\begin{equation}\label{inverse image of qOFM}q_{\bOFM}^{-1}(\bFM^{\s_{k}^{e}-ss}\setminus\bFM^{\s_{k}^{e}-s})=\bOFM_{\fr\ne0}\setminus\bOFM_{\fr\ne0}^{s}.\end{equation}

Thus Theorems \ref{ss FH equiv ss FM}, \ref{st FH equiv st FM in rank 2}, \ref{ss OFH equiv ss OFM} and \ref{st OFH equiv st OFM in rank 2} show that the following diagram commutes with surjective vertical morphisms satisfying (\ref{inverse image of qOFH}) and (\ref{inverse image of qOFM})
$$\xymatrix{\bOFH_{\con,\fr\ne0}\setminus\bOFH_{\con,\fr\ne0}^{s}\ar[r]^{p_{\bOFH}}\ar[d]_{q_{\bOFH}}&\bOFM_{\fr\ne0}\setminus\bOFM_{\fr\ne0}^{s}\ar[d]^{q_{\bOFM}}\\
\bFH_{\con}^{\s_{k}^{e}-ss}\setminus\bFH_{\con}^{\s_{k}^{e}-s}\ar[r]_{p_{\s_{k}^{e}}}&\bFM^{\s_{k}^{e}-ss}\setminus\bFM^{\s_{k}^{e}-s},}$$
which implies that the third statement holds.
\end{proof}

By Theorems \ref{commutativity of full diagram} and \ref{commutativity of C star filps}, we have the following commutative diagram
\begin{equation}\label{FHcontoFM}
{\tiny\xymatrix{&\tbFH_{1,\con}^{-}\ar@/_/[ddd]_{\tp_{1}}\cong\tbFH_{1,\con}^{+}\ar[dl]_{F_{1}^{-}}\ar[dr]^{F_{1}^{+}}&&&&\tbFH_{t,\con}^{-}\ar@/^/[ddd]^{\tp_{t}}\cong\tbFH_{t,\con}^{+}\ar[dl]_{F_{t}^{-}}\ar[dr]^{F_{t}^{+}}&\\
\bFH_{\con}^{\s_{0}-s}\ar[dr]_{G_{1}^{-}}\ar[d]_{\iota_{\bM}}\ar@/_/[ddd]_{p_{\s_{0}}}&&\bFH_{\con}^{\s_{1}-s}\ar[dl]^{G_{1}^{+}}\ar[ddd]^{p_{\s_{1}}}\!\!\!\!\!&\!\!\!\cdots\!\!\!&\!\!\!\!\!\bFH_{\con}^{\s_{t-1}-s}\ar[dr]_{G_{t}^{-}}\ar[ddd]_{p_{\s_{t-1}}}&&\bFH_{\con}^{\s_{t}-s}\ar[dl]^{G_{t}^{+}}\ar@/^/[ddd]^{p_{\s_{t}}}\\
\bM_{\con}&\bFH_{\con}^{\s'_{1}-ss}\ar[ddd]^{p_{\s'_{1}}}&&&&\bFH_{\con}^{\s'_{t}-ss}\ar[ddd]^{p_{\s'_{t}}}&\\
&\tbFM_{1}^{-}\cong\tbFM_{1}^{+}\ar[dl]_{f_{1}^{-}}\ar[dr]^{f_{1}^{+}}&&&&\tbFM_{t}^{-}\cong\tbFM_{t}^{+}\ar[dl]_{f_{t}^{-}}\ar[dr]^{f_{t}^{+}}&\\
\bFM^{\s_{0}-s}\ar[dr]_{g_{1}^{-}}\ar[d]_{\iota_{\bU}}&&\bFM^{\s_{1}-s}\ar[dl]^{g_{1}^{+}}\!\!\!\!\!&\!\!\!\cdots\!\!\!&\!\!\!\!\!\bFM^{\s_{t-1}-s}\ar[dr]_{g_{t}^{-}}&&\bFM^{\s_{t}-s}\ar[dl]^{g_{t}^{+}}\\
\bU(2,d)&\bFM^{\s'_{1}-ss}&&&&\bFM^{\s'_{t}-ss}&}}
\end{equation}

\section{Application}\label{application}
Now we assume that $g\ge2$. We investigate the surjective commutative forgetful diagram from the chain of $\cc^{*}$-flips of $\bFH_{\con}^{\s-ss}(2,d,\cO_{X},\cO_{X})$ to the chain of $\cc^{*}$-flips of $\bFM^{\s-ss}(2,d,\cO_{X})$ of Theorem \ref{commutativity of C star filps} more explicitly.

We first collect some of results on the chain of $\cc^{*}$-flips of $\bFM^{\s-ss}(2,d,\cO_{X})$ which can be verified from \cite{HL95a, Tha94, Tha96}.
\begin{enumerate}
\item[1.] \cite[{cf. (1.3)}]{Tha94} $\bFM^{\s-ss}(2,d,\cO_{X})\ne\emptyset$ if and only if $\s\le-d$.
\item[2.] \cite[{cf. Proposition 2.3}]{HL95a} There exists a discrete set of rationals $0\le\cdots<\eta_{i}<\eta_{i+1}<\cdots\le-d$ including $0$, such that for $\s\in(\eta_{i},\eta_{i+1})$ every $\s$-semistable pair is $\s$-stable, where $\eta_{i}=\max\{0,2i+d\}$. The $\s$-stability conditions depend only on $i$.
\item[3.] \cite[{cf. Proposition 2.9}]{HL95a} Let $\eta_{i}=\max\{0,2i+d\}$, and $$\s'_{1}:=\eta_{\lfloor-d/2-1\rfloor+2}=\begin{cases}1&\text{if }d\text{ is odd}\\2&\text{if }d\text{ is even},\end{cases} \s'_{i}:=\eta_{\lfloor-d/2-1\rfloor+i+1} \text{ and } \s'_{t}:=\eta_{-d-1}=-d-2.$$
\begin{itemize}
\item[(i)] For $0<\s<\s'_{1}$, there is a morphism
$$\bFM^{\s-s}(2,d,\cO_{X})\to\bU(2,d)$$
given by $(E,\psi)\mapsto E$. The fiber over a stable bundle $E$ is isomorphic to $\pp(\Hom(E,\cO_{X})^{\vee})$.
\item[(ii)] Assume that $\s>\s'_{t}$. Then $(E,\psi)\in\bFM^{\s-s}(2,d,\cO_{X})$ if and only if there exists a nonsplit extension
$$\xymatrix{0\ar[r]&F\ar[r]&E\ar[r]^{\psi}&\cO_{X}\ar[r]&0}$$
for some $F\in\Pic^{d}(X)$ with $F=\ker\psi$. Thus
$$\bFM^{\s-s}(2,d,\cO_{X})\cong\pp R^{1}\pi_{\Pic^{d}(X)*}\cP,$$
where $\cP$ is a Poincar\'{e} line bundle on $\Pic^{d}(X)\times X$ and $\pi_{\Pic^{d}(X)}:\Pic^{d}(X)\times X\to\Pic^{d}(X)$ is the projection onto the first factor.
\end{itemize}
\item[4.] \cite[{cf. (2.1)}]{Tha94}
If $(E,\psi)\in\bFM^{\s-s}(2,d,\cO_{X})$, then
\begin{itemize}
\item[(i)] the Zariski tangent space $T_{(E,\psi)}\bFM^{\s-s}(2,d,\cO_{X})$ at $(E,\psi)$ is canonically isomorphic to $H^{1}$ of the complex
$$\xymatrix{C^{0}(\End E)\ar[r]^{p\,\,\,\,\,\,\,\,\,\,\,\,\,\,\,\,}&C^{1}(\End E)\oplus C^{0}(E^{\vee})\ar[r]^{\,\,\,\,\,\,\,\,\,\,\,\,\,\,\,\,\,\,\,\,\,\,\,\,\,q}&C^{1}(E^{\vee})},$$
where $p(u)=(-du,\psi\circ u)$ and $q(v,a)=\psi\circ v+da$;
\item[(ii)] $H^{0}$ and $H^{2}$ of this complex vanish;
\item[(iii)] there is a natural exact sequence
$$0\to H^{0}(\End E)\to H^{0}(E^{\vee})\to T_{(E,\psi)}\bFM^{\s-s}(2,d,\cO_{X})\to H^{1}(\End E)\to H^{1}(E^{\vee})\to0$$
\end{itemize}
\item[5.] \cite[{cf. (2.2)}]{Tha94} If $(E,\psi)\in\bFM^{\s-s}(2,d,\cO_{X})$, then $\dim T_{(E,\psi)}\bFM^{\s-s}(2,d,\cO_{X})=-d+2g-2$.
\item[6.] Let $\bFM^{i}:=\bFM^{\s-s}(2,d,\cO_{X})$ for $\s\in(\max\{0,2i+d\},2i+2+d)$. Note that $0\le\lfloor-d/2-1\rfloor+1\le i\le-d-1$. Let $W_{i}^{-}:=R^{1}\pi_{*}\bM^{-1}(\Delta)$ and $W_{i}^{+}:=R^{0}\pi_{*}\cO_{\Delta}\otimes\bM$, where $\bM$ is the Poincar\'{e} bundle over $\Pic^{i+1}X\times X$, $\Delta\subset S^{-d-i-1}X\times X$ and $\pi:\Pic^{i+1}X\times S^{-d-i-1}X\times X\to\Pic^{i+1}X\times S^{-d-i-1}X$ is the projection onto the first and second factors. Note that $\rk W_{i}^{-}=d+g+2i+1$ and $\rk W_{i}^{+}=-d-i-1$.
\begin{enumerate}
\item \cite[{cf. (3.2)}]{Tha94} There is a family over $\pp W_{i}^{-}$ parametrizing exactly those framed modules which are represented in $\bFM^{i}$ but not $\bFM^{i+1}$.
\item \cite[{cf. (3.3)}]{Tha94} There is a family over $\pp W_{i}^{+}$ parametrizing exactly those framed modules which are represented in $\bFM^{i+1}$ but not $\bFM^{i}$.
\item \cite[{cf. (3.4)}]{Tha94} For $i=-d-2$, the inclusion
$$\iota:\Pic^{-d-1}X\times X\cong\Pic^{-d}X\times X\hookrightarrow\pp R^{1}\pi_{\Pic^{d}(X)*}\cP$$
is given by
$$\iota_{\{M\}\times X}:X\hookrightarrow\pp H^{1}(M^{-1})$$
induced by the linear system $|K_X M|$.
\item \cite[{cf. (3.6)}]{Tha94} $\bFM^{i}$ are all smooth projective varieties of dimension $-d+2g-2$, and for $i<-d-2$, there is a birational map $\bFM^{i}\leftrightarrow\bFM^{\s_{t-1}-s}$, which is an isomorphism except on sets of codimension $\ge2$.
\end{enumerate}
\end{enumerate}

We may assume that $d<0$ by Item 1 mentioned above. Now we investigate the diagram (\ref{FHcontoFM}) in Section \ref{cstar flips} explicitly.

\begin{prop}\label{blowup of FM}
\begin{enumerate}
\item $\bFM^{\s_{i-1}-s}=\bFM^{\lfloor-d/2-1\rfloor+i}$ and $f_{i}^{-}\colon \tbFM_{i}^{-}\to\bFM^{\s_{i-1}-s}$ is the blow-up along $\pp W_{\lfloor-d/2-1\rfloor+i}^{-}$ of codimension $\ge2$ for $i<-d-2$.
\item $\bFM^{\s_{i}-s}=\bFM^{\lfloor-d/2-1\rfloor+i+1}$ and $f_{i}^{+}\colon \tbFM_{i}^{+}\to\bFM^{\s_{i}-s}$ is the blow-up along $\pp W_{\lfloor-d/2-1\rfloor+i}^{+}$ of codimension $\ge2$ for $i<-d-2$.
\item $\tbFM_{i}^{-}\cong\tbFM_{i}^{+}$ for $i<-d-2$.
\item $\bFM^{\s_{t-1}-s}=\bFM^{-d-2}$ is the blow-up of $\bFM^{\s_{t}-s}=\pp R^{1}\pi_{\Pic^{d}(X)*}\cP$ along $\Pic^{-d-1}X\times X$.
\end{enumerate}
\end{prop}
\begin{proof}
\begin{enumerate}
\item It is an immediate consequence from Item 6-(a) and 6-(d).
\item It is an immediate consequence from Item 6-(b) and 6-(d).
\item It is an immediate consequence from \cite[Theorem 3.5]{Tha96}.
\item It is an immediate consequence from Item 6-(c), \cite[Theorem 3.5]{Tha96} and 
$$\dim\pp W_{-d-2}^{-}=d+2g+2i-d-i-1=2g+i-1=-d+2g-3.$$
\end{enumerate}
\end{proof}

We obtain explicit formulas for Poincar\'{e} polynomials of $\bFM^{i},\ \bU(2,d)$ and $\bM_{\con}$ as follows.
\begin{prop}
$\dss P_{t}(\bFM^{i})=-\frac{(1+t)^{2g}}{1-t^{2}}\underset{x^{-d-i-1}}{\coeff}\left(\frac{t^{2d+2g+4i+2}}{1-xt^{4}}-\frac{t^{-2d-2i}}{t^{2}-x}\right)\frac{(1+xt)^{2g}}{(1-x)(1-xt^{2})}$.
\end{prop}
\begin{proof}
By Proposition \ref{blowup of FM},
$$P_{t}(\bFM^{j+1})-P_{t}(\bFM^{j})=P_{t}(\pp W_{j}^{+})-P_{t}(\pp W_{j}^{-}).$$
The RHS is
 \begin{align*}
	P_{t}(\pp W_{j}^{+})-P_{t}(\pp W_{j}^{-}) & =\big(P_{t}(\pp^{-d-j-2})-P_{t}(\pp^{d+g+2j})\big)P_{t}(\Pic^{j+1}X)P_{t}(S^{-d-j-1}X)\\
									    & =\left(\frac{1-t^{2(-d-j-1)}}{1-t^{2}}-\frac{1-t^{2(d+g+2j+1)}}{1-t^{2}}\right)(1+t)^{2g}P_{t}(S^{-d-j-1}X)\\
								            & =\frac{t^{2d+2g+4j+2}-t^{-2d-2j-2}}{1-t^{2}}(1+t)^{2g}P_{t}(S^{-d-j-1}X).
\end{align*}
From $\dss P_{t}(S^{-d-j-1}X)=\underset{x^{-d-j-1}}{\coeff}\frac{(1+xt)^{2g}}{(1-x)(1-xt^{2})}$ by \cite{Mac62}, we have
$$P_{t}(\bFM^{j+1})-P_{t}(\bFM^{j})=\underset{x^{-d-j-1}}{\coeff}\frac{(t^{2d+2g+4j+2}-t^{-2d-2j-2})(1+t)^{2g}(1+xt)^{2g}}{(1-t^{2})(1-x)(1-xt^{2})}.$$
And so
$$\underset{x^{0}}{\coeff}\frac{(t^{-2d+2g-2}-1)(1+t)^{2g}(1+xt)^{2g}}{(1-t^{2})(1-x)(1-xt^{2})}=\frac{(t^{-2d+2g-2}-1)(1+t)^{2g}}{(1-t^{2})}=-P_{t}(\bFM^{-d-1})$$
gives
{\small
\begin{align*}
	P_{t}(\bFM^{i}) & =-\sum_{j=i}^{-d-1}\underset{x^{-d-j-1}}{\coeff}\frac{(t^{2d+2g+4j+2}-t^{-2d-2j-2})(1+t)^{2g}(1+xt)^{2g}}{(1-t^{2})(1-x)(1-xt^{2})}\\
			       & =-\frac{(1+t)^{2g}}{1-t^{2}}\underset{x^{-d-i-1}}{\coeff}\sum_{j=i}^{-d-1}\frac{x^{j-i}(t^{2d+2g+4j+2}-t^{-2d-2j-2})(1+xt)^{2g}}{(1-x)(1-xt^{2})}\\
			       & =-\frac{(1+t)^{2g}}{1-t^{2}}\underset{x^{-d-i-1}}{\coeff}\left(\frac{t^{2d+2g+4i+2}(1-(xt^{4})^{-d-i})}{1-xt^{4}}-\frac{t^{-2d-2i-2}(1-(xt^{-2})^{-d-i})}{1-xt^{-2}}\right)\frac{(1+xt)^{2g}}{(1-x)(1-xt^{2})}\\
			       & =-\frac{(1+t)^{2g}}{1-t^{2}}\underset{x^{-d-i-1}}{\coeff}\left(\frac{t^{2d+2g+4i+2}-t^{-2d+2g+2}x^{-d-i}}{1-xt^{4}}-\frac{t^{-2d-2i}-x^{-d-i}}{t^{2}-x}\right)\frac{(1+xt)^{2g}}{(1-x)(1-xt^{2})}\\
			      & =-\frac{(1+t)^{2g}}{1-t^{2}}\underset{x^{-d-i-1}}{\coeff}\left(\frac{t^{2d+2g+4i+2}}{1-xt^{4}}-\frac{t^{-2d-2i}}{t^{2}-x}\right)\frac{(1+xt)^{2g}}{(1-x)(1-xt^{2})}.
\end{align*}	
}		       
\end{proof}

\begin{prop}\label{Poincare polynomial of U(2,d)}
$\dss P_{t}(\bU(2,d))=\frac{(1+t)^{2g}((1+t^{3})^{2g}-t^{2g}(1+t)^{2g})}{(1-t^{2})(1-t^{4})}$ for odd $d$.
\end{prop}
\begin{proof}
When $-d>2g-2$ is odd, there is a surjective morphism $\bFM^{\s-s}(2,d,\cO_{X})\to\bU(2,d)$ for $0<\s<\s'_{1}$ with fiber $\pp H^{0}(E^{\vee})$. If moreover $-d>4g-4$, then $H^{1}(E^{\vee})=0$ for all stable $E$, so $\bFM^{\s-s}(2,d,\cO_{X})$ is the $\pp^{-d-2g+1}$-bundle over $\bU(2,d)$, and
$$P_{t}(\bU(2,d))=\frac{1-t^{2}}{1-t^{2(-d-2g+2)}}P_{t}(\bFM^{\s-s}(2,d,\cO_{X})).$$

For simplicity, we may assume that $-d=4g-3$. Then $$\s'_{1}=\eta_{\lfloor-d/2-1\rfloor+2}=\eta_{\lfloor(4g-5)/2\rfloor+2}=\eta_{2g-1},\ \bFM^{\s-s}(2,d,\cO_{X})=\bFM^{2g-2}$$ 
and
{\small
\begin{align*}
	P_{t}(\bU(2,d)) & =-\frac{(1+t)^{2g}}{1-t^{4g-2}}\underset{x^{2g-2}}{\coeff}\left(\frac{t^{2g}}{1-xt^{4}}-\frac{t^{4g-2}}{t^{2}-x}\right)\frac{(1+xt)^{2g}}{(1-x)(1-xt^{2})}\\
		               & =\frac{(1+t)^{2g}}{1-t^{4g-2}}\underset{x^{2g-2}}{\coeff}\left(\frac{t^{2g}}{xt^{4}-1}-\frac{t^{4g-2}}{x-t^{2}}\right)\frac{(1+xt)^{2g}}{(1-x)(1-xt^{2})}=\frac{(1+t)^{2g}((1+t^{3})^{2g}-t^{2g}(1+t)^{2g})}{(1-t^{2})(1-t^{4})}
\end{align*}
}
from the proof of \cite[(4.2)]{Tha94}.
\end{proof}
\begin{prop}
$\dss P_{t}(\bM_{\con})=\frac{(1+t)^{2g}((1+t^{3})^{2g}-t^{2g}(1+t)^{2g})}{(1-t^{2})^{2}}$ for odd $d$.
\end{prop}
\begin{proof}
It is an immediate consequence from Theorem \ref{commutativity of full diagram}-(1) and Proposition \ref{Poincare polynomial of U(2,d)}.
\end{proof}

\section*{Acknowledgments}
The authors would like to thank Alexander H. W. Schmitt for his helpful responses to our inquiries. The first named author was supported by Basic Science Research Program through the National Research Foundation of Korea(NRF) funded by the Ministry of Education(No. RS-2023-00241086).

\end{document}